\documentclass[11pt,a4paper]{article}
\usepackage{amsmath}
\usepackage{amsfonts}
\usepackage{amssymb}
\usepackage{amsthm}
\usepackage{amscd}
\usepackage{graphicx}
\usepackage{multirow}
\usepackage[colorlinks=true,linkcolor=blue,citecolor=blue]{hyperref}

\newcommand{\C}[1][]{\ensuremath{{\mathbb{C}^{#1}} }}
\newcommand{\R}[1][]{\ensuremath{{\mathbb{R}^{#1}} }}
\renewcommand{\S}[1][]{\ensuremath{{\mathbb{S}^{#1}} }}
\renewcommand{\H}[1][]{\ensuremath{{\mathbb{H}^{#1}} }}
\newcommand{\Q}[1][]{\ensuremath{{\mathbb{Q}^{#1}} }}

\newcommand{\M}{{\cal M}}

\newcommand{\<}{\langle}
\renewcommand{\>}{\rangle}
\newcommand{\ga}{\gamma}
\newcommand{\pa}{\partial}
\newcommand{\de}{\delta}

\newcommand{\eps}{\epsilon}

\newcommand{\ka}{\kappa}
\newcommand{\la}{\lambda}

\newtheorem{lemm}{Lemma}[section]

\title{Surfaces with one constant principal curvature in three-dimensional space forms}

\author{Henri Anciaux\thanks{supported by CNPq (PQ 306154/2011-0) and Fapesp (2011/21362-2)} 
       }

\date{}

\begin{document}

\maketitle

\begin{abstract}
\noindent We study surfaces with one constant principal curvature  in Riemannian and Lorentzian three-dimensional space forms.  Away from umbilic points they are characterized as one-parameter foliations by curves of  constant curvature, each of them being centered at a point of a  regular curve and contained in its normal plane.
In some cases, a kind of trichotomy phenomenon is observed: the curves of the foliations may be  circles, hyperbolas or horocycles,  depending of whether the constant principal curvature is respectively larger, smaller of equal to one (not necessarily in this order). We describe explicitly some examples showing that there do exist complete surfaces with one constant principal curvature enjoying both umbilic and non-umbilic   points.
\end{abstract}

\bigskip

\centerline{\small \em 2010 MSC: 53C42
\em }

\section*{Introduction}

Surfaces of $3$-dimensional pseudo-Riemannian manifolds having one principal curvature constant constitute a  natural subclass of the class of  Weingarten surfaces. 

\medskip

The  study of such  surfaces has been addressed for the first time in \cite{ST}, where it was proved that a complete orientable surface of Euclidean $3$-space $\mathbb E^3$ with one constant principal curvature equal to $r^{-1}, r > 0$, must be a round sphere  or a tube  over a regular curve, both of radius $r$. 
In particular, a complete, connected surface of $\mathbb E^3$ with one constant principal curvature is either totally umbilic or  free of umbilic points. A tube of radius $r$ over a regular curve $\ga$, called its \em generating curve, \em is the set of points which lie at distance $r$ from $\ga$. In particular, a tube is  foliated by round circles of  radius $r$, each of them being centered at a point of $\ga$ and contained in its normal space. The result holds exactly in the same form in the case of surfaces with one constant principal curvature in the $3$-sphere $\S^3$, although there does not seem to be any account of this fact in the literature.  

\medskip

It turns out that the situation is quite richer in other pseudo-Riemannian space forms, either if the curvature is negative (e.g.\ the hyperbolic space), or if the induced metric is indefinite (e.g.\ in the Minkowski space). 
In \cite{AlGa1}, surfaces of hyperbolic $3$-space $\H^3$ with one constant curvature larger than one have been characterized, while the same result is obtained, by the same method, in \cite{AlGa2}  in the case of spacelike surfaces with one constant curvature smaller than one in $3$-dimensional de Sitter space $d\S^3$. It is also observed, through the description of a couple of examples, that a surface of $\H^3$  with one constant principal curvature smaller than or equal to one, as well as a spacelike surface of $d\S^3$  with one constant principal curvature larger than or equal to one, needs not to be free of umbilic points.

\medskip

The purpose of this paper is to give a description of surfaces having one constant principal curvature  and no umbilic points in the six possible Riemannian or Lorentzian space forms. The general picture is that these surfaces may be regarded as "generalized" tubes, in the sense that they are one-parameter foliations by curves of  constant curvature. Each of these curves are centered at a point of a generating curve and contained in its normal plane.
In some cases, like in $\H^3$ or for spacelike surfaces in $d\S^3$, a kind of trichotomy phenomenon is observed: the curves of the foliations may be  circles, hyperbolas or horocycles,  depending of whether the curvature is respectively larger, smaller of equal to one (not necessarily in this order). 

\medskip

We mention an analogous problem which has been widely studied, namely that of real hypersurfaces in complex space forms whose Hopf field is a  principal direction (see \cite{IR}). The analogy comes from the fact that  such hypersurfaces (called \em Hopf hypersurfaces\em) have the property that the principal curvature associated to the Hopf field must be constant.  In some cases Hopf hypersurfaces are characterized as tubes over complex submanifolds (\cite{CR},\cite{Mo}), and a similar phenomenon of trichotomy, not existing  in complex projective space,  is observed  in the complex hyperbolic space (\cite{IR}). This fact actually has served as an inspiration for the present work.
We also mention that surfaces of $\mathbb E^3$ having one constant principal curvature appear, quite unexpectedly,  in  a problem of convex geometry, the \em Blaschke-Lebesgue problem \em (see \cite{AnGu}): the regular parts of the boundary of a convex body of $\mathbb E^3$ which minimizes the volume among bodies of constant width have one constant principal curvature. Moreover,  the normal congruence (i.e.\ the set of normal geodesics)  of a surface having one constant principal curvature is a marginally trapped Lagrangian surface (see \cite{An}).

\section{Notations and statement of the Main Theorem}
Up to rescaling and anti-isometry, there exist six simply connected,  pseudo-Riemannian manifolds of dimension $3$ and constant sectional curvature.
 For our purposes, it is convenient to describe them as hyperquadrics of $\R^4$ endowed with one of the canonical pseudo-Riemannian metrics, that will be denoted by $\<.,.\>$. 
More precisely, for $p \in \{ 0,1,2\}$, we set
$$\R^4_p :=\left(\R^4,  -\sum_{i=1}^p dx_i^2 + \sum_{i=p+1}^4 dx_i^2 \right).$$
Setting $||x||^2:=\<x,x\>$, we define
\begin{itemize}
\item[-] $\Q^3_{0,0}=\{ x \in \R^4_0 \, | \, x_4 =0\} = \mathbb E^3$ the Euclidean $3$-space;
\item[-] $\Q^3_{1,0}=\{ x \in \R^4_1 \, | \, x_4 =0\}=\mathbb L^3$  the Minkowski $3$-space, i.e.\ the space $\R^3$ equipped with the flat Lorentzian metric;
\item[-] $\Q^3_{0,1}=\{  x \in \R^4_0 \, |  \, \, ||x||^2 =1    \} =\S^3$  the $3$-sphere;
\item[-] $\Q^3_{1,-1}=\{  x \in \R^4_1 \, | \, \, ||x||^2 =-1    \} =\H^3$  the hyperbolic space;
\item[-] $\Q^3_{1,1}=\{  x \in \R^4_1 \, | \, \, ||x||^2 =1    \} =d\S^3$  the de Sitter space;
\item[-] $\Q^3_{2,-1}=\{  x \in \R^4_2 \, | \, \, ||x||^2 =-1    \} =Ad\S^3$  the anti de Sitter space;
\end{itemize}
For technical reasons, it is convenient to introduce as well the space
$$\Q^3_{2,1}=\{  x \in \R^4_2 \, | \, \, ||x||^2 =1    \} =\widetilde{Ad\S^3},$$
which has signature $(-,-,+)$ and is anti-isometric to $Ad\S^3$ through the involution $(x_1,x_2,x_3,x_4) \mapsto (x_3,x_4,x_1,x_2)$ of $\R^4_2$.
In all cases, $\Q^3_{p,\eps}$ is one of these six pseudo-Riemannian space forms, with curvature 
$\eps \in \{-1,0,1\}$ and signature $(p',3-p')$, where $p' := p$ if $\eps \in \{0,1\}$ and $p'=p-1$ if $\eps=-1$.
 Moreover, the immersion $\Q^3_{p,\eps} \to \R^4_p$ is totally umbilical. It follows that given two tangent vectors fields $X,Y$ on $\Q^3_{p,\eps}$, 
the Gauss formula takes the form
\begin{equation} \label{un} D_X Y =\nabla_X Y -\eps \<X,Y\>x ,\end{equation}
where $D$ denote the flat connection of $\R^4_p$ and $\nabla$ the Levi-Civita connection of $\Q^3_{p,\eps}$.
We shall also need to consider the null cones 
$ {\cal NC}_p := \{ x \in \R^4_p \, | \, \, ||x||^2=0\},$ for $p \in \{1,2\}.$
The induced metric on ${\cal NC}_p$ is degenerate since $ x \in x^{\perp} = T_x {\cal NC}_p$.

\bigskip

In the non-flat case $\eps \neq 0,$ the concept of \em polar surface \em is helpful:
 let $\varphi: \M \to \Q^3_{p,\eps}$ an immersed surface with unit normal vector $N$ and set $\eps':=||N||^2.$  
The map $N: \M \to \Q^3_{p,\eps'}$  is an immersion if and only if no principal curvature of $\varphi$ vanishes.  In this case, $N$ is called the \em polar surface \em of $\varphi$. We point out two elementary facts: \textit{(i)}  $\varphi$ is the polar of $N$ and \textit{(ii)} if $\ka$ is a principal curvature of $\varphi$, then $\ka^{-1}$ is a principal curvature of $N.$ It follows that  $\varphi$ has one constant principal curvature if and only if so does $N$. Depending on the curvature of $\Q^3_{p,\eps}$ and on the causal character of $\varphi$, its polar $N$ may be valued in the same space $\Q^3_{p,\eps}$, or in the space $\Q^3_{p,-\eps}$, that we call for this reason the \em polar space \em 
of $\Q^3_{p,\eps}$. For example, the polar of a surface in $\H^3$ is a spacelike surface of $d\S^3$ and conversely. 
This, and the elementary fact \textit{(ii)} stated above, show that
 both the classification result and the examples of \cite{AlGa2} can be deduced from the corresponding classification result and examples of \cite{AlGa1}.

\bigskip

\textbf{Main theorem}

\medskip

 \em

Let $\Sigma$ be a connected, oriented surface of  $\Q_{p,\eps}^3$ which is  free of  umbilic points and such that
  one of its principal curvatures is constant equal to $r^{-1}, \, r >0$. 
Denote by $N$ its unit normal vector and set $\eps':=||N||^2.$  
 Set  $\eps'':=1$
(resp.\ $\eps'':=-1$) if the directions of the constant principal curvature are  spacelike (resp.\ timelike). The \em critical constant \em  of $\Sigma$
is the real number $c:= \eps''  ( \eps+ \eps'  r^{-2} ).$  Then:
\begin{itemize}

\item[-] If $c>0$,   there exists a regular curve $\ga$ in $\Q^3_{p,\eps \eps' \eps''}$  such that $\Sigma$ may be parametrized by the following immersion
$$\varphi(u,v) := |1 +\eps \eps'r^2|^{-1/2} \Big(  \ga(u) + r \big( \cos (v) e_1(u)  + \sin(v) e_2 (u) \big)  \Big),$$
where $(e_1,e_2)$ is an orthonormal frame of the normal space of $\ga$ (which has definite induced metric).
If $\eps'' =\eps'$, $\Sigma$  is an \em elliptic tube, \em i.e.\ the set of  points at distance  
$d$ from  $\ga \in \Q_{p,\eps}^3$, where 
$d := \left\{r, \tan^{-1} (r), \tanh^{-1} (r) \right\}$ if $\eps \eps'=\{0,1,-1\}$ respectively. If $\eps''=- \eps'$ (which implies $\eps'' =\eps$ and $r > 1$), then $\Sigma$  is the polar of  an elliptic tube of $\Q^3_{p,-\eps}$.

\item[-] If  $c <0 $, then $\Sigma$  is a \em hyperbolic tube, \em i.e.\  there exists a regular curve $\ga$ in $\Q^3_{p,-\eps\eps'\eps''}$  such that $\Sigma$
 may be parametrized by the following immersion
$$\varphi(u,v) :=|1 +\eps \eps'r^2|^{-1/2}   \Big(  \ga(u) + r \big( \cosh (v) e_1(u)  + \sinh(v) e_2 (u) \big)  \Big),$$
where $(e_1,e_2)$ is an orthonormal frame of the normal space of $\ga$  (which has indefinite induced metric) with $-||e_1||^2=||e_2||^2=\eps''$. 

\item[-]  If $c=0 $ (which implies $\eps=-\eps'$ and $r=1$), then $\Sigma$  is a \em parabolic tube, \em i.e.\ there exists a regular curve $\de=(\de_+,\de_-)$ in $\Q_{p,\eps}^3 \times \Q^3_{p,-\eps}$  satisfying $\<\de_{\pm},\de_{\mp}\>=\<\de_{\pm}',\de_{\mp}\> =0,$ such that 
 $\Sigma$ may be parametrized by the following immersion
$$\varphi(u,v) :=   \left(1 -\eps \eps'' \frac{v^2}{2} \right)\de_+(u) + v e_1(u)+\frac{v^2}{2}\de_-(u) ,$$ 
where $||e_1||^2=\eps''$ and $\<e_1,\de_{\pm}\>=\<e_1,\de'_{\pm} \>=0$. In particular, if the  curve $\de_+$ (resp.\ $\de_-$) is regular, then $(\de_{-},e_1)$ (resp.\ $(\de_+,e_1)$) is an orthonormal frame of the normal space of  $\de_+$ in $\Q^3_{p,\eps}$ (resp.\ of $\de_-$ in $\Q^3_{p,-\eps}$). 
Finally, exchanging the r\^oles of $\de_+$ and $\de_-$ yields the polar of $\Sigma$.
\end{itemize}

\em

\bigskip


\vspace{2em}

\hspace{-5.5em}  \begin{tabular}{| c|c |  c | c | c | c  | p{1cm} }
     \hline
     Space  & Caus. charac. 
	     &    \multirow{2}{*}{$(\eps,\eps',\eps'')$} &  \multirow{2}{*}{$d=$} &   \multirow{2}{*}{Parametrization of $\Sigma$} &   \multirow{2}{*}{Curve  in}    \\
     form &  of $\Sigma$
     &   & &  &   
 \\  \hline \hline
   $\mathbb E^3$ &   \multirow{5}{*}{spacelike} &$( 0,1,1)$&     $r$ & $\ga+d(\cos(v)e_1+\sin(v)e_2)$ & $\mathbb E^3$    \\    \cline{1-1}   \cline{1-1} \cline{3-6}  \cline{3-6}
$ \S^3$  &   &  $(1,1,1)$& $\tan^{-1}(r)$ & $\cos(d) \ga+\sin(d)\big(\cos(v)e_1+\sin(v)e_2\big)$     & $\S^3$ \\   \cline{1-1} \cline{3-6}   \cline{1-1} \cline{3-6} \cline{1-1} \cline{3-6}
    \multirow{3}{*}{$ \H^3$}&  & \multirow{3}{*}{ $(-1,1,1)$}  & $\tanh^{-1}(r), r<1$ & $\cosh(d) \ga+\sinh(d)\big(\cos(v)e_1+\sin(v)e_2\big)$ &  $\H^3 $    \\  \cline{4-6}
                             &  &  & $\coth^{-1}(r), r>1$ &$\sinh(d) \ga+\cosh(d)\big(\cosh(v)e_1+\sinh(v)e_2\big)$  & 
 $d\S^3 $, spacelike     \\ \cline{4-6}
     &  & & $r=1$ &  $ \left(1 +\frac{v^2}{2} \right)\de_+ + v e_1 + \frac{v^2}{2}\de_-$   &  $\H^3/d\S^3$  
   \\  \hline  \hline  
    \multirow{3}{*}{${\mathbb L}^3$}&  spacelike & $ (0,-1,1)$  & \multirow{3}{*}{ $r$}  & $\ga+d(\cosh(v)e_1+\sinh(v)e_2)$ & 
 $\mathbb L^3, $  spacelike   \\    \cline{2-3} \cline{5-6}
    &  \multirow{2}{*}{timelike}    &  $(0,1,1)$  & &  $ \ga+d(\cos(v)e_1+\sin(v)e_2)$ &    $\mathbb L^3, $  timelike    \\    \cline{3-3} \cline{5-6}
     & &  $(0,1,-1)$  &     &  $\ga+d(\cosh(v)e_1+\sinh(v)e_2)$ &    $\mathbb L^3, $ spacelike    \\ \hline \hline  
	   \multirow{5}{*}{$d\S^3$ }& \multirow{3}{*}{spacelike}   &   \multirow{3}{*}{$(1,-1,1)$} &$\tanh^{-1}(r), r<1$ &$\cosh(d) \ga+\sinh(d)\big(\cosh(v)e_1+\sinh(v)e_2\big)$  & $d\S^3$, spacelike   \\  \cline{4-6}   
   &&  & $\coth^{-1}(r), r>1$  &$\sinh(d) \ga+\cosh(d)\big(\cos(v)e_1+\sin(v)e_2\big)$ &   $\H^3$     \\ \cline{4-6}   
  &&  & $r=1$ &$  \left(1 +\frac{v^2}{2} \right)\de_+ + v e_1+ \frac{v^2}{2}\de_-$ &    $d\S^3/\H^3$     \\ \cline{2-6}
  &  \multirow{2}{*}{timelike}& $(1,1,1)$ &  \multirow{2}{*}{$\tan^{-1}(r)$}  &$\cos(d) \ga+\sin(d)\big(\cos(v)e_1+\sin(v)e_2\big)$  & $d\S^3$, timelike \\ 
   & &$(1,1,-1)$ &  & $\cos(d) \ga+\sin(d)\big(\cosh(v)e_1+\sinh(v)e_2\big)$ &   $d\S^3$, spacelike    \\ \hline \hline  
    \multirow{4}{*}{$ Ad\S^3$} & spacelike   & $ (-1,-1,1)$  & $\tan^{-1}(r)$ & $\cos(d) \ga+\sin(d)\big(\cosh(v)e_1+\sinh(v)e_2\big)$ & $Ad\S^3$, spacelike  \\   \cline{2-6}
 &   \multirow{3}{*}{timelike}  & $(-1,1,1)$&  $\tanh^{-1}(r), r<1$ & $\cosh(d) \ga+\sinh(d)\big(\cos(v)e_1+\sin(v)e_2\big)$ & $Ad\S^3$, timelike   \\   \cline{3-6}
    &&  $(-1,1,-1) $&  $\coth^{-1}(r), r>1$ & $\sinh(d) \ga+\cosh(d)\big(\cos(v)e_1+\sin(v)e_2\big)$ & $\widetilde{Ad\S^3}$  \\ 
 \cline{3-6}   
    && $ (-1,1,-1)$ &  $r=1$ & $  \left(1 -\frac{v^2}{2} \right)\de_+ + v e_1+ \frac{v^2}{2}\de_-$ & 
$ Ad\S^3 /\widetilde{Ad\S^3}$ 
 \\ 
     \hline
   \end{tabular}

\section{The generating curve}
Let $\varphi: \M \to \Q^3_{p,\eps}$ an immersed, connected, orientable surface in $\Q^3_{p,\eps}$.   We assume that $\varphi$ has no umbilic point and that the shape operator $-dN$ is real diagonalizable.
Hence there exists a  coordinate system $(u,v)$ such that
\begin{eqnarray*}
N_u &=& -\ka_1 \varphi_u \\
N_v &=& -\ka_2 \varphi_v,
\end{eqnarray*} 
where $\ka_1$ and $\ka_2$ are the principal curvatures of $\varphi.$ We assume that $\ka_2:=r^{-1}$ is a positive constant. 
We introduce the coefficients of the first and second fundamental form:
$$ E:= \<\varphi_u, \varphi_u\> \quad \quad G:= \<\varphi_v,\varphi_v\>$$
$$ e:= -\<N_u, \varphi_u\> \quad \quad g:= -\<N_v,\varphi_v\>.$$
In particular, 
$$ \ka_1 := \frac{e}{E} \quad \quad \ka_2:= \frac{g}{G}.$$

Since
$$\pa_v \left( \varphi + r N \right) = 0,$$
there exists a function  $u \mapsto \tilde{\ga}(u) \in \R^4$ of the real variable such that 
$$\tilde{\ga}(u)=\varphi(u,v) + r N(u,v).$$
Moreover, by the assumption that $\varphi$ is free of umbilic points, we have that $\tilde{\ga}'= \pa_u \left( \varphi + r N \right) \neq 0,$ hence
$\tilde{\ga}$ is a regular curve.
In the flat case $\eps=0$, we have $\tilde{\ga} \in \Q^3_{p,0}$. In the non flat cases $\eps= \pm 1,$ we have
$$ ||\tilde{\ga}||^2 = ||\varphi||^2  + 2 r\< \varphi, N\> + r^{2} ||N||^2=
 \eps(1+ \eps \eps' r^{2} )= \eps \eps'r^{2} (\eps+ \eps' r^{-2})=\eps \eps' \eps'' r^2 c.$$
We obtain the \em generating curve \em $\ga$ by normalizing $\tilde{\ga}$:

$$\ga := \left\{ \begin{array}{ll} (\eps' \eps'' \frac{c}{|c|}) r^{-1} |c|^{-1/2} \tilde{\ga} & \quad 
\mbox{ if } ||\tilde{\ga}||^2 \neq 0, \\  &\\
  \eps'' \eps' \tilde{\ga}  & \quad \mbox{ if }   ||\tilde{\ga}||^2 =0. 
\end{array} \right.$$
(the reason for the choice of the factor $\eps \eps'' \frac{c}{|c|} = \pm 1$ will become clear in a moment).

If $\eps \eps' \in\{ 0,1\}$, the rescaled curve is contained in $\Q^3_{p,\eps},$  just like the surface $\varphi(\M)$, while
if $ \eps \eps' =-1,$ we observe a trichotomy phenomenon, in the sense that the situation depends on the sign of $1-r^2$:
\begin{itemize}
\item[-] If $ r <1 $, the rescaled curve $\ga$ belongs to $\Q^3_{p,\eps,}$;  
\item[-] If $ r > 1$,  the  curve
 $\ga$ belongs to  $\Q^3_{p,-\eps}$,  the polar space of  $\Q^3_{p,\eps}$; 
\item[-] If $r=1$  (parabolic case), the curve $\ga$ belongs to the null cone ${\cal NC}_p$.
\end{itemize}

\section{The geodesic foliation}
\begin{lemm} \label{geod}
The integral curves of $\pa_v$ are pre-geodesics of the induced metric~$\varphi^*\<.,.\>$.
\end{lemm}
\begin{proof}
We use Codazzi equation 
$$ g_u=\frac{G_u}{2}  (\ka_1+\ka_2),$$ 
together with the assumption that $\ka_2 =\frac{ g}{G}$ is constant, so that
$$ g_u= \frac{G_u g}{G}=G_u \ka_2.$$
Substracting, we get
$$ 0 = \frac{G_u}{2}(\ka_1-\ka_2),$$ 
which implies that $G_u$ vanish. Hence $\Gamma_{22}^1=-\frac{G_u}{2E} $ vanishes as well.
Therefore $\nabla_{\pa_v}\pa_v $ is collinear to $\pa_v,$ which is the claim.
\end{proof}

If follows from the lemma above that we may reparametrize $\varphi$ in such a way that the 
 integral curves of $\pa_v$ are geodesics parametrized by arclength, i.e.\ $|G|=1$. It is therefore convenient to set $\eps'':=G $. Since the integral curves of $\pa_v$ are now geodesics,  their acceleration vector in $\R^4$ $\varphi_{vv}$ must be contained in the plane spanned by $N$ and $\varphi$. Using Equation \ref{un} and the fact that $\tilde{\ga}=\varphi+ rN$, we deduce

\begin{eqnarray*} \varphi_{vv} &=&     \frac{\<\varphi_{vv},N\>}{||N||^2}N  -  \eps \< \varphi_{v},\varphi_v\> \varphi\\
  &=&\frac{g}{\eps'} N  -  \eps G \varphi\\
&=& G \left(  \ka_2 \, \eps' N - \eps \varphi \right)\\
&=& \eps'' \left(   r^{-1} \eps' \, \frac{\tilde{\ga}-\varphi}{r}   - \eps  \varphi \right)  \\
&=&  \eps' \eps'' r^{-2}\tilde{\ga}-\eps''(\eps+\eps' r^{-2}) \varphi    .
\end{eqnarray*}
 We obtain
\begin{equation}  \label{eq}
\varphi_{vv} +  \eps'' \left(\eps+\eps' r^{-2}\right) \varphi  =    \eps' \eps'' r^{-2}  \tilde{\ga}, \end{equation} 
which, for fixed $u$, is a linear, second order equation in the variable $v$. The solution depends on the sign of
 $c:=  \eps'' \left(\eps+\eps' r^{-2}\right)$, which in turn depends on $\eps, \eps' ,\eps''$ and $r$. Observe that if $\eps \eps'=1$ or $0$, the sign does not depend on $r$, while if $\eps \eps'=-1, $ there is trichotomy, i.e.\ the sign of $c$, and therefore the form of the solution, depend on the value of $r.$

\subsection{The hyperbolic case  $c<0$}
Setting $\la:=(-c)^{1/2}$, Equation \ref{eq} becomes $$\varphi_{vv} -\la^2 \varphi =    \eps' \eps''  r^{-2}\tilde{\ga}.$$
The  general solution is
\begin{equation*}\varphi(u,v) =
  A(u) \cosh (\la v) + B(u) \sinh ( \la v) + C(u),
\end{equation*}
where  $A$ and $B$ are two  $\R^4$-valued functions of the variable $u$ and
\begin{eqnarray*} C &=&  \frac{ \eps' \eps''  r^{-2}}{c}\tilde{\ga}\\
&=&-\frac{ r^{-2}}{c} r (-c)^{1/2} \ga\\
&=& |1+ \eps \eps' r^2|^{-1/2}  \ga . 
\end{eqnarray*}

In order to obtain the claimed formula for 
$\varphi$, we differentiate the expression $\varphi=A \cosh (\la v) + B \sinh ( \la v) +  |1+ \eps \eps' r^2|^{-1/2}  \ga$:
\begin{eqnarray*}\varphi_u &=&
  A' \cosh (\la v) + B'  \sinh ( \la v) +  |1+ \eps \eps' r^2|^{-1/2}  \ga', \\
\varphi_v &=&
  A \, \la \sinh (\la v) + B  \, \la  \cosh ( \la  v) .
\end{eqnarray*}
Using first the fact that $ ||\varphi_v||^2$ is constant, we obtain that $||A||^2=-||B||^2$ and $\<A,B\>=0$. 

Next, from the expression
 $\< \varphi_u, \varphi_v\>=0,$ we deduce that  $\<A, \ga'\>$, $\<B,\ga'\>$ and $\<A',B\>=-\<A,B'\>$ vanish. When $\eps \neq 0,$ observing  that $||\varphi||^2$ is constant, we  furthermore obtain the vanishing of
$\<A, \ga\>$ and $\<B,\ga\>$, while, if $\eps=0$, $A$ and $B$ are obviously $\R^3$-valued. 
In both cases, for fixed $u$, the
hyperbola $v \mapsto A \cosh (\la v) + B \sinh ( \la v) $ lies in the normal space of the curve $\ga$  in $T_{\ga} \Q_{p,\eps}^3$. Finally,  observing that 
$$\eps''=||\varphi_v||^2=-||A||^2 (-c)=||B||^2(-c),$$ 
 we find it  convenient to normalize $(A,B)$, introducing the orthonormal frame 
$(e_1,e_2):=  (-c)^{1/2} (A,B).$
We therefore may conclude,  writing
\begin{eqnarray*}\varphi(u,v) &=& C(u)+
  \la^{-1} \left(e_1(u) \cosh (\la v) + e_2(u) \sinh ( \la v) \right) \\
 &=& |1+ \eps \eps' r^2|^{-1/2}  \ga(u)+  
r |1+ \eps \eps' r^2|^{-1/2}  \Big(e_1(u) \cosh (\la v) + e_2(u) \sinh ( \la v) \Big)\\
&=& |1+ \eps \eps' r^2|^{-1/2} \left(  \ga(u)+  
r \Big(e_1(u) \cosh (\tilde{v}) + e_2(u) \sinh ( \tilde{ v}) \Big) \right) ,\\
\end{eqnarray*}
where we have set $\tilde{v}:=\la v.$ This is the claimed formula. 
As a final comment,  we observe that if we release the condition $\<A',B\>=0$, the coordinate field $\pa_u$ is not anymore principal, but $\pa_v$ is still the principal direction associated to the constant principal curvature $\ka_2 = r^{-1}$.

\subsection{The elliptic case $c >0$}
The elliptic case  is analogous to (and somehow simpler than) the hyperbolic one and left to the Reader.


\subsection{The parabolic case  $c=0$}
Here $r= 1$ and $\eps\eps'=-1$ (observe in particular that the parabolic case does not occur in the flat case $\eps=0$) and Equation (\ref{eq}) becomes
$$ \varphi_{vv} =  \eps' \eps''   \tilde{\ga}=  \ga.$$ 
The general solution is
$$\varphi(u,v) = \frac{1}{2}      \ga(u)  \, v^2+ A(u)v + B(u) ,$$
where $A$ and $B$ are two  $\R^4$-valued functions of the variable $u$.
Differentiating
\begin{eqnarray*}\varphi_u &=& \frac{1}{2}    \ga' \, v^2+ A'v + B', \\
\varphi_v &=&     \ga  \, v+ A,
 \end{eqnarray*}
and using first the facts that $||\varphi||^2=\eps$ and $ ||\varphi_v||^2=\eps''$, we obtain  
\begin{eqnarray*} ||B||^2 &= & \eps\\
||A||^2 &=& \eps'' \\
\<A,B\>&=&0\\
  \<A,\ga\>&=&0\\
||A||^2 +  \<B, \ga\>&=& 0,  \quad \mbox{so} \quad \<B, \ga\>=-\eps''.
\end{eqnarray*}
From $\<\varphi_u,\varphi_v\>=0$ we deduce moreover
\begin{eqnarray*} \<A,B'\> &= & 0\\
   \< A',A\> + \<\ga, B'\> &=& 0,  \quad \mbox{so} \quad \<B, \ga'\>=-\<B',\ga\> =\<A,A'\>=0\\
\frac{1}{2} \<\ga',A\> + \< \ga, A'\>&=&0, \quad \mbox{so} \quad \<A, \ga'\> =0.
 \end{eqnarray*}
We now introduce the orthonormal triple  $(\de_+,\de_-,e_1):=(B,  \eps \eps'' B+\ga,A  )$. The calculations above imply $(||\de_+||^2, ||\de_-||^2, ||e_1||^2)=(\eps, -\eps,-\eps'')$ and $\<\de_{\pm}',e_1\>=\<\de_{\pm}',\de_{\mp}\>=0$. 
Since the curve $\ga = -\eps \eps'' \de_+ + \de_-$ is regular, the velocity vectors $\de_+'$ and $\de_-'$ cannot vanish simultaneously, i.e.\ $(\de_+,\de_-)$ is regular. 
 We therefore obtain the required formula
$$\varphi(u,v) :=   \left(1 -\eps \eps'' \frac{v^2}{2} \right)\de_+(u) + v e_1(u)+\frac{v^2}{2}\de_-(u).$$
To conclude, we see that 
if $\de_+$ (resp.\ $\de_-$) is a regular curve, then  $(e_1,\de_-)$(resp.\ $(e_1,\de_+)$)  is an orthonormal  frame of the normal space of $\de_+$ (resp.\ $\de_-$). 


\section{Some examples}

\subsection{An exemple in $\mathbb E^3$}
We first describe a very simple example showing that the completeness assumption in the classification result of \cite{ST} is necessary.  Using the decomposition ${\mathbb E^3} := \R \times \C$, we set:
\begin{eqnarray*}
e_0(u) & := &  ( 0,i e^{iu })\\
e_1(u) & := & ( 0, e^{iu })\\
e_2(u) & := &  ( 1,0 )
\end{eqnarray*}
and define the immersion:
$$ \varphi(u,v) := \int_0^u h(s) e_0(s) ds \,  +r \Big( \cos (v) e_1(u) + \sin (v) e_2(u) \Big),$$
  where $r >0$ and $h \in C^\infty(\R)$. 
According to the Main Theorem, the  principal curvature $\ka_2$ of $\varphi$ is constant equal to $r^{-1}$. A straightforward calculation shows that the other principal curvature of $\varphi$ takes the form
$$ \ka_1(u,v)= \frac{\cos(v)}{h(u)  + r \cos(v)  }.$$
Observe that $\varphi$ fails to be an immersion if $\inf_{\R} |h| \leq r$. On the other hand, $\varphi$ is umbilic  at the points where $h$ vanishes, which are exactly the points where the curve $\ga(u):=\int_0^u h(s) e_0(s) ds$ fails to be regular. An interval where $h$ vanishes corresponds to  a vertex of the trajectory of $\ga$.
 Around such a point, the image of $\varphi$ consists of two (non-complete) "half-tubes", smoothly connected by a portion of sphere.

\subsection{An exemple in $\mathbb L^3$} \label{L3}
A slight modification of the previous example shows that, unlike the $\mathbb E^3$ case,  there exist complete surfaces in $\mathbb L^3$ with one principal constant curvature which are not free of umbilic points.
The frame $(e_0,e_1,e_2)$ defined in the previous section is still orthornomal in ${\mathbb L^3} := \R \times \C$, with $||e_2||^2=-1$, and  we
 define the immersion:
$$ \varphi(u,v) := \int_0^u h(s) e_0(s) ds +r \Big( \cosh (v) e_1(u) + \sinh (v) e_2(u) \Big),$$
  where $r >0$ and $h \in C^\infty(\R)$. Again, the principal curvature $\ka_2$ is constant equal to $r^{-1}$, while the other principal curvature is given by
$$ \ka_1(u,v)= \frac{\cosh(v)}{h(u)  + r \cosh(v)  },$$
so $\varphi$ is umbilic  at the vanishing points of $h$. On the other hand, $\varphi$ is a complete immersion if and only if $\inf_{\R}  h >- r$, so, with a suitable choice of  $h$, we obtain a complete immersion with the claimed property.
 The difference with the previous case comes from the non-compactness of the hyperbolas, which allows, at the vertices of  $\ga(u):=\int_0^u h(s) e_0(s) ds$, to have two complete hyperbolic half-tubes smoothly connected by a portion of hyperboloid.

\subsection{The examples of Aledo and G\'alvez}
In \cite{AlGa1} and \cite{AlGa2}, other examples of complete surfaces with one constant principal curvature enjoying both umbilic and non-umbilic   points are described implicitely through their first and second fundamental forms. We explain how these examples can be recovered from our construction.

\medskip

The first example  is a surface of $\H^3$  with one constant principal curvature $\ka_2=1$, while the other principal curvature satisfies
$$\ka_1(u,v)=1-\frac{h(u)}{2+v^2 h(u)},$$
where $h: \R \to \R$ is some non-negative function. The corresponding surface is umbilic at the  points where $h$ does not vanish. According to our Main Theorem, if  $\de_+$ is a regular curve of $\H^3$ and $(e_1,\de_-)$ is an orthonormal frame of  the normal space of $\de_+$ in $T_{\de_+} d\S^3$, then the
 immersion 
$$\varphi(u,v) = \left(1 +\frac{v^2}{2} \right)\de_+(u) + v e_1(u)+ \frac{v^2}{2}\de_-(u)$$
 as one principal curvature $\ka_2$ constant equal to $1$.
A routine calculation shows that the other principal curvature of $\varphi$ is given by 
$$\ka_1(u,v)=1 - \frac{1-b(u)}{1- a(u) v +(1-b(u))\frac{v^2}{2}  },$$
where $a:=\< \de_+'', e_1\> $ and $b:= \< \de_+'', \de_-\> $. It is easily seen that the surface becomes umbilic whenever $b=1$.
We obtain the example of Aledo and G\'alvez taking $a=0$  and $h=1-b$, so in particular $\de_-$ and $e_1$ are respectively the normal and binormal vectors of $\de_+$, and $b$ is the curvature of $\de_+$. In this case the set of umbilic points of $\varphi$   correspond to points of the curve $\de_+$ with curvature $1$. 

\medskip

The second example of  \cite{AlGa1} resembles that described in Section \ref{L3}: it is
  a surface of $\H^3$  with one constant principal curvature 
$R <1$, while the other principal curvatures satisfies, in   a local coordinate system $(x,y)$,
\begin{equation} \label{Gal} \frac{h(y) + R(1-R^ 2) e^{\sqrt{1-R^2}y}}{R \, h(x)+ (1-R^2) e^{\sqrt{1-R^2}y}}, 
\end{equation}
where $h: \R \to \R$ is some non-negative function.
 The corresponding surface is umbilic at  points where $h$ vanishes. According to our Main Theorem, such surface may be parametrized, at least away from umbilic points,  by an immersion of the form
$$\varphi(u,v) =(r^2-1)^{-1/2} \left( \ga+ r\big(\sinh(v)e_1+\cosh(v)e_2\big) \right) ,$$
where $\ga$ is a regular, spacelike curve of $d\S^3$ and $(e_1,e_2)$ is an orthonormal frame of  the normal space of $\ga$ in $T_\ga d\S^3$.
We recover Aledo and G\'alvez's example considering $\ga,e_1$ and $e_2$ in such a way that the orthonormal moving frame $ u \mapsto F=(\ga, e_0,e_1,e_2) \in SO(3,1)$ is a solution of the differential system
$$ F^{-1} F' = \left( \begin{array}{cccc} 0 & \tilde{h} & 0 &0 \\
               -\tilde{h}  &0& -1 & 1 \\ 0& 1 &0 &0 \\ 0 & 1 & 0 &0 \end{array} \right).$$
A short calculation gives the following expression for the corresponding principal curvature $\ka_1$:
\begin{equation} \label{Moi} \ka_1(u,v)=\frac{r \tilde{h}(u) + e^{v}}{\tilde{h}(u)+r e^{v}}, 
\end{equation}
which is exactly Formula (\ref{Gal}), setting $\tilde{h}:=(1-R^2)h, \,  v:=\sqrt{1-R^2} \, y$ and $r:=R^{-1}$. Since $\ga'(u)=h(u)e_0(u)$,  the umbilic points of $\varphi$ correspond once again to vertices of the curve $\ga$.

Henri Anciaux \\
Universidade de S\~ao Paulo \\
	Instituto de Matem\'atica e Estat\'istica \\ 
	  Rua do Mat\~ao, 1010, \\ S\~ao Paulo, 05508-090, Brazil \\
	e-mail: \emph{henri.anciaux@gmail.com}
\end{document}